\RequirePackage[l2tabu, orthodox]{nag}

\documentclass[
11pt,                          
english                        
]{article}

\usepackage[english]{babel}    
\usepackage{amsmath}           
\usepackage[utf8]{inputenc}    
\usepackage[T1]{fontenc}       
\usepackage{longtable}         
\usepackage{exscale}           
\usepackage[final]{graphicx}   
\usepackage[sort]{cite}        
\usepackage{array}             
\usepackage{wasysym}           
\usepackage[a4paper]{geometry} 
\usepackage{gitinfo2}          
\usepackage[multiuser]{fixme}  
\usepackage{xspace}            
\usepackage{tikz}              
\usepackage{ifdraft}           
\usepackage{chairx}            
\usepackage[expansion=false    
           ]{microtype}        
\usepackage[nottoc]{tocbibind} 
\usepackage[
           final=true,         
           pdfpagelabels       
           ]{hyperref}         

\graphicspath{{../tikz/}}

\usetikzlibrary{matrix}
\usetikzlibrary{arrows}
\usetikzlibrary{patterns}
\usetikzlibrary{decorations.pathreplacing}
\usetikzlibrary{calc}

\ifdraft{\synctex=1}{}

\fxusetheme{color}

\FXRegisterAuthor{bs}{anbs}{bastian}
\FXRegisterAuthor{sw}{answ}{stefan}

\author{
  \textbf{Bastian Seifert}\thanks{\texttt{Bastian.Seifert@hs-ansbach.de}}\\[0.3cm]  
    Julius Maximilian University of Würzburg \\
    Department of Mathematics \\
    97074 Würzburg \\
    Germany\\[0.2cm]
    and \\[0.2cm]
    Ansbach University of Applied Sciences \\
    Faculty of Engineering Sciences \\
    91522 Ansbach\\
    Germany \\[1cm]
}

\DeclarePairedDelimiter{\gen} {\langle}{\rangle}

\newcommand{\Lmod}[1][R]{#1\text{-}\mathsf{Mod}}

\newcommand{\cargument}{\,-\,}

\newcommand{\Pic}{\mathsf{Pic}}

\newcommand{\cotensor}[1][{}]{\mathbin{\square_{\scriptscriptstyle{#1}}}}

\newcommand{\Coend}{\mathsf{CoEnd}}
\newcommand{\Cohom}{\mathsf{CoHom}}

\newcommand{\Rcomod}[1][C]{\mathsf{CoMod}\text{-}#1}

\newcommand{\coacts}{\delta}
\newcommand{\Coacts}{\circlearrowright}
\newcommand{\co}{\mathrm{co}}

\newcommand{\counit}{\varepsilon}
\renewcommand{\conv}{\star}

\newcommand{\injco}[1]{\widehat{#1}}

\renewcommand{\conv}{\star}

\title{Equivariant Morita-Takeuchi Theory}

\begin{document}

%
%

\maketitle

%
%

\begin{abstract}
    We introduce the notion of $H$-equivariant Morita-Takeuchi theory
    for coalgebras with symmetries given by a Hopf algebra $H$. A
    cohomology theory is introduced which classifies the possible
    lifts of coactions on coalgebras to corresponding comodules. An
    equivariant Picard groupoid is defined and its connection to the
    developed cohomology theory investigated.
\end{abstract}

%
%

\tableofcontents
\newpage

%
%

\section{Introduction}
\label{sec:Introduction}%

In physical theories one is interested in the classification of
observable algebras only up to measuring equivalence. Measuring
equivalence in a strict sense is implemented by various types of
Morita equivalence, like the classical theory of
Morita~\cite{morita:1958a} where one realizes the measuring
equivalence as equivalence of module categories, the more realistic
approach of Rieffel~\cite{rieffel:1974b} for $C^*$-algebras, or the
algebraic essence of Rieffel's theory for $^*$-algebras by Bursztyn
and Waldmann~\cite{bursztyn.waldmann:2001a,bursztyn.waldmann:2001b}.

One is interested in taking care of symmetries, aswell. Probably one
of the best ways to construct symmetries is by Hopf algebra actions.
They incorporate symmetries modeled by groups and Lie algebras and
additionally are the natural habitat of quantum groups. One then needs
to lift the action of the Hopf algebra on the observable algebra to
the module categories used in the various Morita theories. For the
algebraic $^*$-Morita theory this was done in
\cite{jansen.waldmann:2006a} by Jansen and Waldmann in the setting of
equivariant Morita theory.

Another interesting thing to study are coalgebras, the dual of
algebras in the sense of inverting arrows. In the case of coalgebras
one is interested in co-measuring equivalence in the sense of
equivalence of comodule categories, aswell. The first approaches to
study this Co-Morita theory were developed by Lin~\cite{Lin:1974} and
Takeuchi~\cite{Takeuchi:1977}. The theory of Takeuchi turned out to be
more potent than the the theory of Lin. For coalgebras over rings the
most reliable theory was developed by
Al-Takhman~\cite{AlTakhman:2002a,AlTakhman:2002b}, which is the
approach on which we rely on in this article.

To implement symmetries in the setting of coalgebras it is plausible
to rely on coactions of Hopf algebras instead of actions. In
Section~\ref{sec:CoactionsHopfAlgebras} we develop the theory of Hopf
algebra coactions in the specific way we need for equivariant
Morita-Takeuchi theory. We define a cohomology theory based on the
introduced schism complex which allows for the classification of
possible lifts of coactions on coalgebras to coactions on
corresponding comodules. Special cases of the appearing cohomology
groups are the cocharacter group and the coinvariants. A notion of
$H$-equivariant coderivations is investigated, which injects into the
cocharacters by a convolution exponential. The developed cohomology
theory is somewhat dual to the Sweedler cohomology. Hence it can be
interpreted as the cohomology theory exponential to an analogon of the
Hochschild cohmology in the coalgebra setting.

In Section~\ref{sec:EquivariantMoritaTakeuchi} we proceed by equipping
Morita-Takeuchi theory with coactions in an equivariant setting. We
show that the cotensor product is compatible with $H$-coactions. We
classify all possible Hopf algebra coactions, which give rise to an
equivariant Morita-Takeuchi bicomodule, by the first cohomology group
of the schism complex.

All appearing coalgebras, Hopf algebras etc. are over the unital ring $R$,
which is assumed to be commutative, of characteristic zero, a
principal ideal domain, local and noetherian.

%
%

%
%



%
%

\section{Coactions of  Hopf algebras and their cohomology}
\label{sec:CoactionsHopfAlgebras}%

We are interested in symmetries of rather different types, e.g. Lie
algebra actions or group actions. Hence we would like to consider
symmetries of high generality. To achieve this, we implement the
notion of symmetry as \emph{Hopf algebras}.

Recall that a Hopf algebra is an algebra which has a compatible
coalgebra structure, where the unit has an inverse with respect to the
convolution product, the antipode. We will denote the algebra
multiplication by $\mu$, the unit by $\id$, the comultiplication by
$\Delta$, the counit by $\counit$ and the antipode by $S$. We will use
Sweedler notation
\begin{equation}
    \label{eq:SweedlerNotation}
    \Delta(h) = h_{(1)} \tensor h_{(2)}
\end{equation}
to denote comultiplication.

In this section we investigate how Hopf algebras can be interpreted as
symmetries. The classical way of thinking about symmetries is that of
an action. By an action we take a point of e.g.\ an algebra and move
it under the influence of an element of the Hopf algebra. But we could
consider the dual situation, i.e. that of a coaction of a Hopf
algebra, aswell. Here one is interested in how a point of a coalgebra
can be reached by elements of a Hopf algebra or more figurative, how
an element of a coalgebra decays.

Unlike in many parts of the literature, where one considers only
module or comodule structures of a Hopf algebra as action or coaction,
respectively, we want to keep track of the additional structures
appearing in a Hopf algebra. That is instead of only a module we would
consider additionally some flatness condition on it, which ensures
compatibility with the comultiplication. And instead of only a
comodule we consider some coflatness condition, which ensures
compatibility with the product of the Hopf algebra. The terms flat and
coflat are not to be considered in the categorical sense, but instead
should be thought of as coming from differential geometry, where one
considers \emph{flat connections}, i.e. those giving rise to flat
bundles. Thanks to the Serre-Swan theorem, flat bundles correspond to
flat modules. Hence the flatness conditions ensures that we can lift
the action to the Morita modules. Analogously the coflatness
conditions ensures that we have coflat modules, i.e.\ we can lift the
coaction to Morita-Takeuchi comodules.

The structure of this section is as follows. First we define the
notion of a Hopf coaction, give some examples and clarify the notion
of coinvariants. Then we will study the corresponding
\emph{cocharacter groups}, which are the analogue to character groups
of actions. We show how one can associate a Lie algebra of
coderivations to the cocharacter groups. Finally we include the
cocharacter groups in the bigger context of a cohomology theory for
Hopf algebras.

\subsection{Coactions of Hopf algebras}
\label{subsec:HopfCoaction}%

In this subsection we investigate the notion dual to that of an
action, i.e.\ the notion of coactions. These coactions describe how an
element of a coalgebra decays. We first state the definition, which at
first glance might look a bit messy, but afterwards the definition is
clarified by an extended version of Sweedler notation. The axioms are
just the categorical duals of that of an action. We use the flip
morphism $\tau$ with permutation in the superscript, indicating the
elements which are switched.
\begin{definition}[Hopf coaction]
    \label{definition:HopfCoaction}%
    Let $C$ be a coalgebra and $H$ a Hopf algebra. A (right) coaction of $H$
    on $C$ is a linear map $\coacts \colon C \to C \tensor H$ which
    satisfies the following axioms:
    \begin{definitionlist}
        \item \label{item:CoactionsIsComodule}
        $C$ is a $H$-comodule via $\coacts$, i.e.\ we have
        \begin{equation}
            \label{eq:CoactionCompatibleWithComultiplication}
            (\coacts \tensor \id) \circ \coacts
            = (\id \tensor \Delta_H) \circ \coacts
        \end{equation}
        and
        \begin{equation}
            \label{eq:CoactionCounitary}
            c \tensor 1 = (\id \tensor \counit)(\coacts c)
        \end{equation}
        for all $c \in C$.
        \item \label{item:CoactionIsCoflat}
        The coaction $\coacts$ is coflat, i.e. we have
        \begin{equation}
            \label{eq:CoactionCoflatnessCompatibleProduct}
            (\id \tensor \mu) \circ \tau^{(2 3)} \circ (\coacts \tensor
            \coacts) \circ \Delta_C
            = (\Delta_C \tensor \id) \circ \coacts
        \end{equation}
        and
        \begin{equation}
            \label{eq:CoactionCoflatnessCounitarity}
            (\counit_C \tensor \counit_H)(\coacts(c))
            = \counit_C(c) \tensor 1
        \end{equation}
        for $c \in C$.
    \end{definitionlist}
\end{definition}
The definition of a left coaction is completely analogously.
\begin{remark}
    \label{remark:ExtendingSweedlerNotationToCoaction}%
    The equations stated in Definition~\ref{definition:HopfCoaction}
    are very abstract. Hence we would like to concretize them using
    Sweedler notation. However we are dealing with different types of
    comultiplications and corepresentations and later we would like to
    have a good notation for coactions on comodules, aswell. A way to
    solve this issue is the following. We extend the sumless Sweedler
    notation as follows: for $c \in C$ and a coaction $\coacts$ one
    denotes the elements coming from the coaction by \emph{upper}
    indices, i.e. one writes
    \begin{equation}
        \label{eq:ExtendedSweedlerNotationForCoaction}
        \coacts c = c^{(0)} \tensor c^{(1)},
    \end{equation}
    the zero being reserved for elements of the coalgebra.
    Then the axioms of Definition~\ref{definition:HopfCoaction} can be
    stated as follows. For $c \in C$ we have for part
    \refitem{item:CoactionsIsComodule}
    \begin{equation}
        \label{eq:SweedlerCoactionCompatibleWithComultiplication}
        c^{(0)} \tensor c^{(1)} \tensor c^{(2)}
        = c^{(0)} \tensor (c^{(1)})_{(1)} \tensor (c^{(1)})_{(2)},
    \end{equation}
    which allows us to always write the comultiplication of $H$ in the upper
    indices, and
    \begin{equation}
        \label{eq:SweedlerCoactionCounitary}
        c \tensor 1 = c^{(0)} \tensor \counit(c^{(1)}),
    \end{equation}
    and for the second part one has
    \begin{equation}
        \label{eq:SweedlerCoactionCoflatnessCompatibleProduct}
        c_{(1)}^{(0)} \tensor c_{(2)}^{(0)} \tensor c_{(1)}^{(1)}
        c_{(2)}^{(1)}
        = c_{(1)}^{(0)} \tensor c_{(2)}^{(0)} \tensor c^{(1)}
    \end{equation}
    and
    \begin{equation}
        \label{eq:SweedlerCoactionCoflatnessCounitarity}
        \counit(c^{(0)}) \tensor \counit(c^{(1)})
        = \counit(c) \tensor 1.
    \end{equation}
    The axioms ensure the notation is valid, as the ambiguity of which
    comultiplication is used in
    \eqref{eq:SweedlerCoactionCoflatnessCompatibleProduct} disappears
    thanks to
    \eqref{eq:SweedlerCoactionCompatibleWithComultiplication}.
\end{remark}
We give some examples.
\begin{example}
    \label{example:Coactions}%
    \begin{examplelist}
        \item \label{item:TrivialCoaction}
        The trivial coaction is given by simply tensorizing with $1$,
        i.e. if $C$ is a coalgebra and $H$ any Hopf algebra then the
        trivial coaction of $H$ on $C$ is given by
        \begin{equation}
            \label{eq:TrivialCoaction}
            \coacts(c) = c \tensor 1
        \end{equation}
        for all $c \in C$.
        \item \label{item:GradingCoaction}
        Let $H = RG$ be a group Hopf algebra and let $C =
        \bigoplus_{g \in G} C_g$ be a $G$-graded coalgebra. Then the
        grading coaction of $H$ on $C$ is given by
        \begin{equation}
            \label{eq:GradingCoaction}
            \coacts c = c \tensor g
        \end{equation}
        for $c \in C_g$.
        \item \label{item:AdjointCoaction} Let $H$ be a Hopf algebra.
        Considering the underlying coalgebra we can investigate the
        adjoint coaction of $H$ on itself. The adjoint coaction is
        given by
        \begin{equation}
            \label{eq:AdjointCoaction}
            \coacts h = h_{(2)} \tensor S(h_{(1)}) h_{(3)}
        \end{equation}
        for all $h \in H$. For $H$ being a group Hopf algebra this is
        again the trivial coaction, since then
        $\coacts(h) = h \tensor h^{-1} h = h \tensor 1$. This holds
        for every cocommutative Hopf algebra aswell, since one has
        $\coacts(h) = h_{(2)} \tensor S(h_{(1)}) h_{(3)} = h_{(1)}
        \tensor \counit(h_{(2)}) = h \tensor 1$. For Sweedler's Hopf
        algebra
        $H_4 = R\gen{1,g,x,gx \; | \; g^2 = 1, x^2 = 0, xg = - gx}$ one
        gets $\coacts g = g \tensor 1$,
        $\coacts x = 1 \tensor xg + x \tensor g + g \tensor gx$, and
        $\coacts xg = g \tensor gx + gx \tensor g + 1 \tensor xg$.
        \item \label{item:InnerCoaction}
        One can consider an inner coaction, i.e. if we have an Hopf
        algebra $H$, a coalgebra $C$ and a coalgebra homomorphism $J
        \colon C \to H$ then the inner coaction is given by
        \begin{equation}
            \label{eq:InnerCoaction}
            \coacts c = c_{(2)} \tensor S(J(c_{1})) J(c_{(3)})
        \end{equation}
        for all $c \in C$.
    \end{examplelist}
\end{example}
As always we need a corresponding notions of morphisms, so we define
equivariant coalgebra morphisms:
\begin{definition}[$H$-equivariant coalgebra homomorphism]
    \label{definition:CoactionEquivariantMorphisms}%
    Let $H$ be a Hopf algebra coacting on the coalgebras $C$ and $D$.
    A coalgebra homomorphism $\Phi \colon C \to D$ is called
    $H$-equivariant if
    \begin{equation}
        \label{eq:CoactionEquivariantMorphisms}
        \phi(c^{(0)}) \tensor c^{(1)} = \phi(c)^{(0)} \tensor
        \phi(c)^{(1)}
    \end{equation}
    for all $c \in C$.
\end{definition}
Analogously to the invariants of an action one defines the
coinvariants of a coaction as those elements which are affected by the
coaction like under the trivial coaction.
\begin{definition}[Coaction coinvariants]
    \label{definition:CoinvariantsOfCoaction}%
    Let $C$ be a coalgebra with coaction $\coacts$ of $H$. The set of
    coinvariants is denoted by
    \begin{equation}
        \label{eq:CoinvariantsOfCoaction}
        C^{\co H} = \{c \in C \; | \; \coacts(c) = c \tensor 1 \}.
    \end{equation}
\end{definition}
From the definition we can deduce a formula for the coaction
interacting with the antipode of the Hopf algebra.
\begin{lemma}
    \label{lemma:AntipodeAndCoaction}%
    Let $C$ be a coalgebra carrying a coaction of a Hopf algebra $H$.
    Then
    \begin{equation}
        \label{eq:AntipodeAndCoaction}
        c_{(1)} \tensor c_{(2)}^{(0)} \tensor c_{(2)}^{(1)}
        = c_{(1)}^{(0)} \tensor c_{(2)}^{(0)} \tensor S(c_{(1)}^{(1)})
          c^{(1)}
    \end{equation}
    and
    \begin{equation}
        \label{eq:AntipodeAndCoactionSecond}
        c_{(1)}^{(0)} \tensor c_{(2)} \tensor c_{(1)}^{(1)}
        = c_{(1)}^{(0)} \tensor c_{(2)}^{(0)} \tensor c^{(1)}
          S(c_{(2)}^{(1)})
    \end{equation}
    hold for all $c \in C$.
\end{lemma}
\begin{proof}
    This follows by the following calculation
    \begin{align*}
        c_{(1)}^{(0)} \tensor c_{(2)}^{(0)} \tensor S(c_{(1)}^{(1)})
        c^{(1)}
        &\stackrel{(*)}{=} c_{(1)}^{(0)} \tensor c_{(2)}^{(0)} \tensor
           S(c_{(1)}^{(1)}) c_{(1)}^{(1)} c_{(2)}^{(1)} \\
        &= c_{(1)}^{(0)} \tensor c_{(2)}^{(0)} \tensor
           \counit(c_{(1)}^{(1)}) c_{(2)}^{(1)} \\
        &= c_{(1)} \tensor c_{(2)}^{(0)} \tensor c_{(2)}^{(1)},
    \end{align*}
    where we used counitarity, the properties of the antipode and in
    ($*$) the compatibility with the Hopf
    product~\eqref{eq:SweedlerCoactionCoflatnessCompatibleProduct}.
    The second equation follows completely analogously.
\end{proof}

\subsection{Cocharacter groups}
\label{subsec:CocharacterGroups}%

In the representation theory of groups the \emph{characters} play a
fundamental role. For Lie groups also the notion of an infinitesimal
character (or derivations) is useful. We are interested in a theory of
\emph{cocharacters}, i.e. the theory dual to that of characters. This
theory will play an important role in characterizing the kernel of the
forgetful morphism from the equivariant Picard group to the standard
Picard group. In this section we make use of the characteristic zero
condition on $R$, as we need to be able to divide through integers and
hence $\mathbb{Q} \subseteq R$ is crucial.

A nice investigation of character groups of Hopf algebras with trivial
action is given in \cite{bogfjellmo.dahmen.schmedig:2015a}, the
equivariant version, even for a noncommutative target algebra, was
introduced in \cite{jansen.waldmann:2006a}. Basically, one considers
the algebra $\Hom(H, \algebra{A})$ of linear morphisms from the Hopf
algebra $H$ to the algebra $\algebra{A}$ with product being the
convolution product and finds suitable subgroups of this algebra as
character groups.

We now go the other way round and pair a coalgebra $C$ with a Hopf
algebra $H$, that is we consider the set $\Hom(C, H)$ with convolution
product. Recall that for the set of linear maps from a coalgebra $C$
to an algebra $\algebra{A}$ one has the convolution product
\begin{equation}
    \label{eq:ConvolutionProductAgain}%
    f \conv g = \mu \circ (f \tensor g) \circ \Delta
\end{equation}
for $f,g \in \Hom(C,\algebra{A})$. This convolution product endows
$\Hom(C, \algebra{A})$ with the structure of an algebra. For a Hopf
algebra $H$ one can apply this on both sides, since $H$ has a
coalgebra and an algebra structure.

The first part is remodelling the group of which the character groups
are subgroups. We have to replace the unit by the counit. We call a
morphism $\phi \in \Hom(C,H)$ \emph{counitary} if it satisfies
$\counit_H \circ \phi = \counit_C$. We set
\begin{equation}
    \label{eq:CounitarityMonoidCocharacters}
    \Gamma_0 = \{ \phi \in \Hom(C,H) \; | \; \phi \text{ is counitary
      and convolution invertible}\}
\end{equation}
This is completely dual to the situation in the definition of unitary
characters as the unitarity condition $\phi(1) = 1$ is actually
$\phi \circ \unit_H = \unit_{\algebra{A}}$ due to linearity. The
corresponding Lie algebra is
\begin{equation}
    \label{eq:CochcaracterLieAlgebra}
    \lie{c}_0 = \{\phi \in \Hom(C,H) \; | \; \counit_H \circ \phi = 0
      \}.
\end{equation}
Obviously $\Gamma_0$ forms a group with respect to the the convolution
product $\conv$ and $\lie{c}_0$ a Lie algebra with Lie bracket being
the commutator to the convolution product. The $\lie{c}_0$ injects
into $\Gamma_0$ via the convolution exponential
$\exp_\conv(\phi) = \sum_{k \geq 0} \frac{\phi^{\conv k}}{k!}$ if $H$
is a filtered Hopf algebra. Recall that a filtration of a Hopf algebra
$H$ is a filtration of the underlying $R$-module, i.e.\ one has
submodules
$H^0 \subseteq H^1 \subseteq \dots \subseteq H^n \subseteq \dots$ with
$\bigcup_{n \geq 0} H^n = H$, such that one has
$\mu(H^p \tensor H^q) \subseteq H^{p+q}$,
$\Delta(H^n) \subseteq \sum_{p+q = n} H^p \tensor H^q$, and
$S(H^n) \subseteq H^n$. As example of a filtered Hopf algebra consider
the following construction of the coradical filtration. Consider $H$
as $(H,H)$-bicomodule and let $H^0 = \mathsf{soc} \; H$ be the
coradical, i.e.\ the sum of all simple subcoalgebras of $H$. If the
coradical $H^0$ is a Hopf subalgebra then one can recursively define a
filtration of $H$ via
\begin{equation}
    \label{eq:CoradicalFiltration}
    H^n = \Delta^{-1}( H \tensor H^0 + H^{n-1} \tensor H)
\end{equation}
which is a Hopf algebra filtration, see \cite[Thm.\ II.2.2, Rem.\
1]{Manchon.2001}. The condition that $H^0$ is a Hopf subalgebra is for
example fulfilled for every pointed Hopf algebra, i.e.\ Hopf algebras
whose simple Hopf subalgebra are all one-dimensional. These include
all cocommutative Hopf algebras, e.g.\ the group algebras and
universal enveloping algebras of Lie algebras.

Let us denote by $\coacts$ the coaction of $H$ on the coalgebra $C$.
We adapt the extended Sweedler notation from
Remark~\ref{remark:ExtendingSweedlerNotationToCoaction}.
\begin{definition}[$H$-equivariant cocharacters]
    \label{definition:CocharacterGroupOfHopfAlgebra}%
    An element $\phi \in \Hom(C,H)$ is called $H$-equivariant
    cocharacter if it is convolution invertible and satisfies the
    following axioms:
    \begin{definitionlist}
        \item \label{item:CocharacterCounitary}
        $\counit_H \circ \phi = \counit_C$,
        \item \label{item:CocharacterCoactionCondition}
        $\phi(c)^{(1)} \tensor \phi(c)^{(2)} = \phi(c_{(1)}^{(0)})
        \tensor c_{(1)}^{(1)} \phi(c_{(2)})$,
        \item \label{item:CocharacterComoduleCondition}
        $c_{(1)}^{(0)} \tensor c_{(1)}^{(1)} \phi(c_{(2)})
        = c_{(2)}^{(0)} \tensor \phi(c_{(1)}) c_{(2)}^{(1)}$,
    \end{definitionlist}
    for all $c \in C$. The set of all $H$-equivariant cocharacters is
    denoted by
    \begin{equation}
        \label{eq:CocharacterGroup}
        \Gamma(C \Coacts H)
        = \{ \phi \in \Hom(C,H) \; | \; \phi \text{ is } H
        \text{-equivariant cocharacter} \}.
    \end{equation}
\end{definition}
Note that the conditions can be phrased completely in the terms of
morphisms as follows: for \eqref{item:CocharacterCoactionCondition}
one has
\begin{equation}
    \label{eq:CocharacterCoactionConditionInMorphisms}
    \Delta_H \circ \phi
    = (\id \tensor \mu) \circ (\phi \tensor \id \tensor \phi) \circ
    (\coacts \tensor \id) \circ \Delta_C
\end{equation}
and for \eqref{item:CocharacterComoduleCondition} one has
\begin{equation}
    \label{eq:CocharacterComoduleConditionInMorphisms}
    (\id \tensor \mu) \circ (\coacts \tensor \phi) \circ \Delta_C
    = (\id \tensor \mu) \circ \tau^{(1 2)} \circ (\phi \tensor \coacts)
      \circ \Delta_C.
\end{equation}
These conditions arise by dualizing the conditions of
\cite[Definition~A.1]{jansen.waldmann:2006a}. We call condition
\refitem{item:CocharacterCounitary} counit condition, condition
\refitem{item:CocharacterCoactionCondition} the coaction condition and
condition \refitem{item:CocharacterComoduleCondition} the comodule
condition.
\begin{proposition}
    \label{proposition:CocharactersFormGroup}%
    The set $\Gamma(C \Coacts H)$ forms a group. The inverse of $\phi
    \in \Gamma(C \Coacts H)$ is given by
    \begin{equation}
        \label{eq:CocharacterInverse}
        \phi^{\conv -1}(c) = S(\phi(c^{(0)})) c^{(1)}.
    \end{equation}
\end{proposition}
\begin{proof}
    Let $\phi, \psi \in \Gamma(C \Coacts H)$. We have to check that
    $\phi \conv \psi \in \Gamma(C \Coacts H)$. The counit condition is
    clear. For the coaction condition we have for $c \in C$
    \begin{align*}
        (\phi \conv \psi)(c)_{(1)} \tensor (\phi \conv \psi)(c)_{(2)}
        &= \phi(c_{(1)})_{(1)} \psi(c_{(2)})_{(1)} \tensor
           \phi(c_{(1)})_{(2)} \psi(c_{(2)})_{(2)} \\
        &= (\mu \tensor \mu) \circ \tau^{(2 3)} (\phi(c_{(1)})_{(1)}
           \tensor \phi(c_{(1)})_{(2)} \tensor \psi(c_{(2)})_{(1)}
           \tensor \psi(c_{(2)})_{(2)}) \\
        &\stackrel{(a)}{=} (\mu \tensor \mu) \circ \tau^{(2 3)}
           (\phi(c_{(1)}^{(0)}) \tensor c_{(1)}^{(1)} \phi(c_{(2)})
           \tensor \psi(c_{(3)}^{(0)}) \tensor c_{(3)}^{(1)}
           \psi(c_{(4)})) \\
        &= \phi(c_{(1)}^{(0)}) \psi(c_{(3)}^{(0)}) \tensor
           c_{(1)}^{(1)} \phi(c_{(2)}) c_{(3)}^{(1)} \psi(c_{(4)}) \\
        &\stackrel{(b)}{=} \phi(c_{(1)}^{(0)}) \psi(c_{(2)}^{(0)})
           \tensor c_{(1)}^{(1)} c_{(2)}^{(1)} \phi(c_{(3)})
           \psi(c_{(4)}) \\
        &\stackrel{(c)}{=} (\phi \conv \psi)(c_{(1)}^{(0)}) \tensor
        c_{(1)}^{(1)} (\phi \conv \psi)(c_{(2)})
    \end{align*}
    where in ($a$) we used the coaction condition
    \refitem{item:CocharacterCoactionCondition} for $\phi$ and $\psi$,
    in ($b$) the comodule condition
    \refitem{item:CocharacterComoduleCondition}, and in ($c$) we used
    \eqref{eq:SweedlerCoactionCoflatnessCompatibleProduct}. Checking
    the comodule condition is comparatively easy:
    \begin{align*}
        c_{(1)}^{(0)} \tensor c_{(1)}^{(1)} (\phi \conv \psi)(c_{(2)})
        &= c_{(1)}^{(0)} \tensor c_{(1)}^{(1)} \phi(c_{(2)})
           \psi(c_{(3)}) \\
        &\stackrel{\mathclap{\refitem{item:CocharacterComoduleCondition}}}{=}
           c_{(2)}^{(0)} \tensor \phi(c_{(1)}) c_{(2)}^{(1)}
           \psi(c_{(3)}) \\
        &\stackrel{\mathclap{\refitem{item:CocharacterComoduleCondition}}}{=}
            c_{(3)}^{(0)} \tensor \phi(c_{(1)}) \psi(c_{(2)})
            c_{(3)}^{(1)} \\
        &=  c_{(2)}^{(0)} \tensor (\phi \conv \psi)(c_{(1)})
            c_{(2)}^{(1)},
    \end{align*}
    for $c \in C$.

    The candidate for the inverse satisfies
    \begin{align*}
        \phi^{\conv -1} \conv \phi(c)
        &= S(\phi(c_{(1)}^{(0)})) c_{(1)}^{(1)} \phi(c_{(2)}) \\
        &\stackrel{(a)}{=} S(\phi(c)_{(1)}) \phi(c)_{(2)} \\
        &\stackrel{(b)}{=} e(\phi(c)) \\
        &\stackrel{(c)}{=} e(c),
    \end{align*}
    where in ($a$) we used the coaction condition of $\phi$, in ($b$)
    the fact that $S$ is the convolution inverse to the identity,
    applied via $S \conv \id(\phi(c))$, and in ($c$) we used the
    counitarity of $\phi$. Now we have to check that $\phi^{\conv -1}
    \in \Gamma(C \Coacts H)$. The counitarity of $\phi^{\conv -1}$ is
    clear by definition.
    For the comodule condition, we calculate
    \begin{align*}
        c_{(1)}^{(0)} \tensor c_{(1)}^{(1)} \phi^{\conv -1}(c_{(2)})
        &= c_{(1)}^{(0)} \tensor c_{(1)}^{(1)} S(\phi(c_{(2)}^{(0)}))
           c_{(2)}^{(1)} \\
        &\stackrel{(a)}{=} c_{(1)}^{(0)} \tensor c_{(1)}^{(1)}
           S(\phi(c_{(2)}^{(0)})) S(c_{(1)}^{(1)}) c^{(1)} \\
        &\stackrel{(b)}{=} c_{(1)}^{(0)} \tensor c_{(1)}^{(2)}
           S(c_{(1)}^{(1)} \phi(c_{(2)}^{(0)})) c^{(1)} \\
        &\stackrel{(c)}{=} c_{(2)}^{(0)} \tensor c_{(1)}^{(1)}
           S(\phi(c_{(1)}^{(0)}) c_{(2)}^{(1)}) c^{(1)} \\
        &= c_{(2)}^{(0)} \tensor c_{(1)}^{(1)} S(c_{(2)}^{(1)})
           S(\phi(c_{(1)}^{(0)})) c^{(1)} \\
        &= c_{(2)}^{(0)} \tensor S(\phi(c_{(1)}^{(0)})) c^{(1)} \\
        &\stackrel{(d)}{=} c_{(2)}^{(0)} \tensor
           S(\phi(c_{(1)}^{(0)})) c_{(1)}^{(1)} c_{(2)}^{(1)} \\
        &= c_{(2)}^{(0)} \tensor \phi^{\conv -1}(c_{(1)})
           c_{(2)}^{(1)}
    \end{align*}
    where in ($a$) we used the identity
    \eqref{eq:AntipodeAndCoaction}, in ($b$) we used the properties of
    the antipode, in ($c$) the comodule condition for $\phi$ and in
    ($d$) the compatibility of the coaction with product of the Hopf
    algebra.
    Finally we have to check the coaction condition. Here we have
    \begin{align*}
        \phi^{\conv -1}(c_{(1)}^{(0)}) \tensor c_{(1)}^{(1)}
        \phi^{\conv - 1}(c_{(2)})
        &= S(\phi(c_{(1)}^{(0)})) c_{(1)}^{(1)} \tensor c_{(1)}^{(2)}
           S(\phi(c_{(2)}^{(0)})) c_{(2)}^{(1)} \\
        &\stackrel{(a)}{=} S(\phi(c_{(1)}^{(0)})) c_{(1)}^{(1)}
           c_{(2)}^{(1)} \tensor S(c_{(1)}^{(1)}) c_{(1)}^{(2)}
           S(\phi(c_{(2)}^{(0)})) c_{(2)}^{(2)} \\
        &= S(\phi(c_{(1)}^{(0)})) c^{(1)} \tensor
           S(\phi(c_{(2)}^{(0)})) \counit(c_{(1)}) c_{(2)}^{(1)} \\
        &= S(\phi(c_{(1)}^{(0)})) c^{(1)} \tensor
           S(\phi(c_{(2)}^{(0)})) S(c_{(1)}^{(1)}) c_{(1)}^{(2)}
           c_{(2)}^{(2)} \\
        &= S(\phi(c_{(1)}^{(0)})) c^{(1)} \tensor S(c_{(1)}^{(1)}
           \phi(c_{(2)}^{(0)})) c^{(2)} \\
        &\stackrel{(b)}{=} (S(\phi(c^{(0)}))^{(1)} c^{(1)} \tensor
           (S(\phi(c^{(0)})))^{(2)} c^{(2)} \\
        &= \phi^{\conv -1}(c)^{(1)} \tensor \phi^{\conv -1}(c)^{(2)}
    \end{align*}
    where we used in ($a$) the fact that $\id \conv S = e$, whence
    $\counit \circ \Delta = (\id \tensor S) \circ \Delta$, in ($b$)
    the coaction condition for $\phi$ and that $S$ is a coalgebra
    homomorphism, and the properties of antipode and counit.
\end{proof}
\begin{example}
    \label{example:TrivialCocharacterOnGrouplike}%
    Consider the a group-like coalgebra $R[S]$ for some set $S$ and a
    Hopf algebra $H$. Assume the trivial coaction $\coacts \colon R[S]
    \to R[S] \tensor H$ given by $\coacts s = s \tensor 1$. Then the
    cocharacter group $\Gamma(C \Coacts H)$ is given by those
    morphisms $\phi \in \Hom(C, H)$ which send group-like elements to
    group-like elements, as the coaction condition in this case reads
    \begin{equation}
        \label{eq:TrivialCocharacterOnGrouplike}
        \phi(s)^{(1)} \tensor \phi(s)^{(2)} = \phi(s) \tensor \phi(s)
    \end{equation}
    for all $s \in S$.
\end{example}

There is a corresponding infinitesimal notion of coderivations:
\begin{definition}[$H$-equivariant coderivations]
    \label{definition:EquivariantCoderivations}%
    An element $\phi \in \Hom(C,H)$ is called $H$-equivariant
    coderivation if it is convolution invertible and satisfies the
    following:
    \begin{definitionlist}
        \item \label{item:EquivCoderUnit}
        $\counit_H \circ \phi = 0$
        \item \label{item:EquivCoderCoaction}
        $\phi(c)^{(1)} \tensor \phi(c)^{(2)}
        = e(c_{(1)}^{(0)}) \tensor c_{(1)}^{(1)} \phi(c_{(2)}) +
          \phi(c_{(1)}^{(0)}) \tensor c_{(1)}^{(1)} e(c_{(2)})$
        \item \label{item:EquivCoderComodule}
        $c_{(1)}^{(0)} \tensor c_{[1)}^{(1)} \phi(c_{(2)})
        = c_{(2)}^{(0)} \tensor \phi(c_{(1)}) c_{(2)}^{(1)}$,
    \end{definitionlist}
    for all $c \in C$. The set of all $H$-equivariant coderivations is
    denoted by $\lie{c}(C \Coacts H)$.
\end{definition}
\begin{proposition}
    \label{proposition:EquivCoderivationExpToCocharacters}%
    If $H$ is filtered, the set $\lie{c}(C \Coacts H)$ injects into
    $\Gamma(C \Coacts H)$ via $\exp(\phi) = \sum_{k = 0}^\infty
    \frac{\phi^{\conv k}}{k!}$.
\end{proposition}
\begin{proof}
    The filtration of $H$ is needed to ensure convergence of the
    exponential map.

    We first have to show that 
    \begin{equation*}
        \phi^{\conv n}(c)^{(1)} \tensor \phi^{\conv n}(c)^{(2)}
        = \sum_{k=0}^n \binom{n}{k} \phi^{\conv k}(c_{(1)}^{(0)})
        \tensor c_{(1)}^{(1)} \phi^{\conv(n-k)}(c_{(2)})
    \end{equation*}
    holds for all for $\phi \in \lie{c}(C \Coacts H)$ and all
    $c \in C$. This is done by induction, for $n=1$ it is just the
    infinitesimal coaction condition, and for $n \to n+1$ we have
    \begin{align*}
        \Delta_H \circ \phi^{\conv n+1}(c)
        &= \mu( \Delta_H(\phi(c_{(1)})) \tensor \Delta_H(\phi^{\conv
           n}(c_{(2)}))) \\
        &= \sum_{k=0}^n \binom{n}{k} (e(c_{(1)}^{(0)}) \tensor
           c_{(1)}^{(1)} \phi(c_{(2)}) + \phi(c_{(1)}^{(0)}) \tensor
           c_{(1)}^{(1)} e(c_{(2)})) (\phi^{\conv k}(c_{(3)}^{(0)})
           \tensor c_{(3)}^{(1)} \phi^{\conv n-k}(c_{(4)})) \\
        &= \sum_{k=0}^n \binom{n}{k} (\phi^{\conv k}(c_{(1)}^{(0)})
           \tensor c_{(1)}^{(1)} \phi^{\conv n+1-k}(c_{(2)}) +
           \phi^{\conv k+1}(c_{(1)}^{(0)}) \tensor c_{(1)}^{(1)}
           \phi^{\conv n-k}(c_{(2)})) \\
        &= \sum_{k=0}^n \binom{n}{k} \phi^{\conv k}(c_{(1)}^{(0)})
           \tensor c_{(1)}^{(1)} \phi^{\conv n+1-k}(c_{(2)}) +
           \sum_{k=1}^{n+1} \binom{n}{k-1} \phi^{\conv
             k}(c_{(1)}^{(0)}) \tensor c_{(1)}^{(1)} \phi^{\conv n -
             (k+1)}(c_{(2)}) \\
        &= c_{(1)}^{(1)} \phi^{\conv n+1}(c_{(2)}) + \sum_{k=1}^{n+1}
           \binom{n+1}{k} \phi^{\conv k}(c_{(1)}^{(0)}) \tensor
           c_{(1)}^{(1)} \phi^{\conv n+1-k}(c_{(2)}) + \phi^{\conv
             n+1}(c_{(1)}^{(0)}) \\
        &= \sum_{k=0}^{n+1} \binom{n+1}{k} \phi^{\conv
            k}(c_{(1)}^{(0)}) \tensor c_{(1)}^{(1)} \phi^{\conv
            n+1-k}(c_{(2)}).
    \end{align*}
    This now allows to show that the exponential of the infinitesimal
    coaction condition is indeed the coaction condition:
    \begin{align*}
        \exp(\phi)(c)^{(1)} \tensor \exp(\phi)(c)^{(2)}
        &= \sum_{n \geq 0} \frac{1}{n!} \phi^{\conv n}(c)^{(1)}
           \tensor \phi^{\conv n}(c)^{(2)} \\
        &= \sum_{n \geq 0} \frac{1}{n!} \sum_{k=0}^n \binom{n}{k}
           \phi^{\conv k}(c_{(1)}^{(0)}) \tensor c_{(1)}^{(1)}
           \phi^{\conv n-k}(c_{(2)}) \\
        &= \sum_{k=0}^\infty \sum_{n=k}^\infty \frac{1}{k! (n-k)!}
           \phi^{\conv k}(c_{(1)}^{(0)}) \tensor c_{(1)}^{(1)}
           \phi^{\conv n-k}(c_{(2)}) \\
        &= \sum_{k=0}^\infty \frac{1}{k!} \phi^{\conv k}(c_{(1)}^{(0)})
           \tensor \sum_{n=0}^\infty \frac{1}{n!} c_{(1)}^{(1)}
           \phi^{\conv n}(c_{(2)}) \\
        &= \exp(\phi)(c_{(1)}^{(0)}) \tensor c_{(1)}^{(1)}
           \exp(\phi)(c_{(2)}).
    \end{align*}
    The counit and comodule condition follow immediately, hence the
    injection is clear.
\end{proof}

\subsection{The schism complex}
\label{sec:HopfCohomology}%

Recall that the first Hochschild cohomology group classifies the
derivations up to inner derivations. Hence one might wonder, if one
can define a cohomology theory for Hopf algebras, which has the
character groups as 1-cocycles and is sort of exponential to the
Hochschild cohomology as we can exponentiate the derivations to the
characters via the convolution exponential. This is done for
cocommutative Hopf algebras with coefficients in a commutative algebra
via \emph{Sweedler cohomology} see \cite{Sweedler:1968}. In the
noncommutative setting there are attempts to generalize this
cohomology theory. The definition of the first noncommutative
cohomology group is straight-forward but the second one has some more
serious obstructions to tackle. Some attempts can be found in the
articles of Schauenburg~\cite{Schauenburg:2001}, and Bichon and
Carnovale~\cite{Bichon.Carnovale:2006}.

In this subsection we are interested in how one can implement a
similar setting for cocharacter groups. This is we build a theory
which observes how one can split a coalgebra under the presence of a
coaction of a Hopf algebra $H$ and then tries to conclude something
about the Hopf algebra. As in the case of cohomology we are not
interested the whole cocharacter group, but only the ones, which are
not inner ones. These will then be defined as first homoschism group.
Note that schism is the greek word for split. The zeroth homoschism
group will be defined to coincide with the coinvariants of the Hopf
coaction on the coalgebra.

We first define the general zeroth and first homoschism group, and
then embedd them in the commutative, cocommutative case into a bigger
setting somewhat dual to that of Sweedler cohomology for Hopf
algebras~\cite{Sweedler:1968}. The definition of the
noncommutative homoschism groups is motivated by the definition of the
zeroth and first noncommutative homology groups of Hopf algebras.

Recall that a coalgebra morphism $f \colon C \to D$ is cocentral if it
cocommutes with the identity, that is
$f(c_{(1)}) \tensor c_{(2)} = f(c_{(2)}) \tensor c_{(1)}$ for all
$c \in C$. The cocenter of a coalgebra is the subcoalgebra missing the
elements which spoil cocommutativity, i.e.\
$\Zentrum(C) = C \big/ \bigcap_{f \text{ cocentral}} \ker f$ see
\cite{Torrecillas.Zhang:1996} for a more concrete construction.
\begin{definition}[Noncommutative homoschism groups]
    \label{definition:ZerothAndFirstCohomoschismGroup}%
    Let $H$ be a Hopf algebra and let $C$ be a coalgebra with coaction
    of $H$. The zeroth homoschism group of $H$ with covalues in $C$ is
    the group of coinvariants
    \begin{equation}
        \label{eq:ZerothHomoschismGroup}
        S^0(H,C) = \Zentrum(C)^{\co H}.
    \end{equation}
    The first homoschism group of $H$ with covalues in $C$ is the
    group of cocharacters modulo the inner ones
    \begin{equation}
        \label{eq:FirstHomoschismGroup}
        S^1(H,C) = \Gamma(C \Coacts H) \big/ \injco{\Zentrum(C)},
    \end{equation}
    where $\injco{\argument} \colon \Zentrum(C) \to \Gamma(C \Coacts
    H)$ is given by
    \begin{equation}
        \label{eq:CoinjectionCocenterIntoCocharacters}
        \injco{c} = ((\delta_c \tensor \id) \circ \coacts) \conv
        (\delta_c \tensor 1)^{\conv -1}
    \end{equation}
    for all $c \in \Zentrum(C)$ and $\delta_c \in \Zentrum(C)^*$ is
    the corresponding delta functional.
\end{definition}
One can mod out the image of the cocenter, since by dualizing one gets
a central subgroup. As we will see later in
Section~\ref{sec:EquivariantMoritaTakeuchi} the first Homoschism group
allows to classify coaction lifts to comodules.

We are now interested embedding these two groups, at least in the
commutative, cocommutative setting, into a larger theory. For the
Sweedler cohomology one defines various groups of linear maps going
from tensor powers of the Hopf algebra to the algebra. So the idea is
to define groups of linear maps going from the coalgebra to tensor
powers of the Hopf algebra. These tensor powers of the Hopf algebra
carry an algebra structure.

Recall that the differential for the Sweedler cohomology is defined by
succesively combining elements of the Hopf algebra in various manners
and then sending it to the algebra. As the name schism complex
suggests, one would suspect to have the differential between the
groups defined by succesively splitting the coalgebra elements and
then sending them to the Hopf algebra. For the splitting process we
can either use the coaction or the comultiplication of the Hopf
algebra. Hence it is natural to consider all possible combinations of
these.

We justified the following definition.
\begin{definition}[Schism complex]
    \label{definition:SchismComplex}%
    Let $H$ be a commutative Hopf algebra and $C$ be a cocommutative
    coalgebra. The schism complex consists of the sequence
    $\Gamma_0^\bullet(C,H)$ of groups defined by
    $\Gamma_0^q(C,H) = \Gamma_0(C, H^{\tensor q})$ and the
    differential, which is defined for
    $\D^{q-1} \colon \Gamma_0^{q-1}(C,H) \to \Gamma_0^{q}(C, H)$ by
    \begin{align}
        \label{eq:SchismDifferential}
        \nonumber
        \D^{q-1}(f)
        &= ((f \tensor \id) \circ \coacts) \conv ((\Delta_H \tensor
        \id \tensor \dots \tensor \id) \circ f)^{\conv - 1} \conv ((\id
        \tensor \Delta_H \tensor \id \tensor \dots \tensor \id) \circ
        f) \conv  \\
        &\quad \dots \conv ((\id \tensor \dots \tensor \id \tensor
        \Delta_H) \circ f)^{\conv \pm 1} \conv (f \tensor 1)^{\conv \mp 1}
    \end{align}
    for $f \in \Gamma_0^{q-1}(C,H)$. The schism complex is denoted by
    $(\Gamma_0^\bullet(C, H), \D^\bullet)$.
\end{definition}
This is indeed a cochain complex:
\begin{theorem}
    \label{theorem:SchismDifferentialIsDifferential}%
    If $H$ is a commutative Hopf algebra, we have $\D \circ \D = e$
    for the schism differential.
\end{theorem}
\begin{proof}
    Thanks to commutativity of $H$ and cocommutativity of $C$ the
    convolution $\conv$ is commutative. We want to show, that we can
    apply $\D$ to each term separately, i.e. one has $\D(f \conv g) =
    \D(f) \conv \D(g)$ for all $f,g \in \Gamma_0^{q-1}(C \Coacts H)$,
    $q \geq 1$. Since $\Gamma_0^q(C,H)$ is abelian, the only
    problematic part is the one concerning the coaction. So let $f,g
    \in \Gamma_0^{q-1}(C,H)$ and $c \in C$ then
    \begin{align*}
        ((f \conv g) \tensor \id)(c^{(0)} \tensor c^{(1)})
        &= f(c_{(1)}^{(0)}) g(c_{(2)}^{(0)}) \tensor c^{(1)} \\
        &\stackrel{(*)}{=} f(c_{(1)}^{(0)}) g(c_{(1)}^{(0)}) \tensor
           c_{(1)}^{(1)} c_{(2)}^{(1)} \\
        &= ((f \tensor \id) \circ \coacts) \conv ((g \tensor \id)
           \circ \coacts) (c)
    \end{align*}
    where in ($*$) we used the coflatness of the coaction.

    Now we can check $\D^q \circ \D^{q-1} = e$. So let $f \in
    \Gamma_0^{q-2}(C,H)$ then one has
    \begin{align*}
        \D^{q-1}(f)
        &= ((f \tensor \id) \circ \coacts) \conv ((\Delta_H \tensor
        \id \tensor \dots \tensor \id) \circ f)^{\conv - 1} \conv ((\id
        \tensor \Delta_H \tensor \id \tensor \dots \tensor \id) \circ
        f) \conv  \\
        &\quad \dots \conv ((\id \tensor \dots \tensor \id \tensor
        \Delta_H) \circ f)^{\conv \pm 1} \conv (f \tensor 1)^{\conv \mp 1}.
    \end{align*}
    Thanks to $\D(f \conv g) = \D(f) \conv \D(g)$, we can now apply
    $\D^q$ to each term separately and get for the first term, using
    the notation $\Delta_H^i = (\id \tensor \dots \tensor \Delta_H
    \tensor \dots \tensor \id)$ with $\Delta_H$ at the $i$-th position,
    \begin{align*}
        \D^q((f \tensor \id) \circ \coacts)
        &= ((((f \tensor \id) \circ \coacts) \tensor \id) \circ
           \coacts) \conv (\Delta_H^1
           \circ ((f \tensor \id) \circ \coacts))^{\conv -1} \conv
           \dots \\
        &\quad \dots \conv (\Delta_H^{q} \circ ((f \tensor \id) \circ
           \coacts))^{\conv \pm  1} \conv (((f \tensor \id) \circ
           \coacts) \tensor 1)^{\conv \mp 1} \\
        &= ((f \tensor \id) \circ (\id \tensor \Delta_H) \circ
           \coacts) \conv (\Delta_H^1 \circ ((f \tensor \id) \circ
           \coacts))^{\conv -1} \conv \dots \\
        &\quad \dots \conv (\Delta_H^{q} \circ ((f \tensor \id) \circ
           \coacts))^{\conv \pm  1} \conv (((f \tensor \id) \circ
           \coacts) \tensor 1)^{\conv \mp 1}
    \end{align*}
    for the $i$-th middle term
    \begin{align*}
        \D^q(( \Delta_H^i \circ f)^{\conv \pm 1})
        &= (((\Delta_H^i \circ f)^{\conv \pm 1} \tensor \id) \circ \coacts)
           \conv   (\Delta_H^1 \circ (\Delta_H^i \circ f)^{\conv
           \pm 1})^{\conv - 1} \conv \dots \\
        &\quad \dots \conv (\Delta_H^q \circ (\Delta_H^i \circ
           f)^{\conv \pm 1})^{\conv \mp 1} \conv ((\Delta_H^i \circ
           f)^{\conv \pm 1} \tensor 1)^{\conv \pm 1}
    \end{align*}
    and for the end term
    \begin{align*}
        \D^q((f \tensor 1)^{\conv \mp 1})
        &= ((f \tensor 1)^{\conv \mp 1} \circ \coacts) \conv
           (\Delta_H^1 \circ (f \tensor 1)^{\conv \mp 1})^{\conv -1}
           \conv \dots \\
        &\quad \dots \conv (\Delta_H^q \circ (f \tensor 1)^{\conv \mp
           1})^{\conv \mp 1} \conv ((f \tensor 1)^{\mp 1} \tensor
           1)^{\conv \pm 1}
    \end{align*}
    Now with shrewd eyes one sees that for each term there is one with
    the opposite sign, hence we get $\D \circ \D = e$.
\end{proof}
Next we investigate the cohomology groups of the schism complex.
\begin{definition}[Homoschism groups]
    \label{definition:HomoschismGroups}%
    The homoschism groups for a commutative Hopf algebra $H$ with
    covalues in a cocommutative coalgebra $C$ is defined by
    \begin{equation}
        \label{eq:ZerothHomoschismGroupComplex}
        S^0(C,H) = \ker \D^0
    \end{equation}
    and
    \begin{equation}
        \label{eq:HigherHomoschismGroups}
        S^q(C,H) = \ker \D^q \big/ \image \D^{q-1}
    \end{equation}
    for $q \geq 1$.
\end{definition}
We have to check that in the commutative case both definitions of
$S^0(H,C)$ and $S^1(H,C)$ coincide.
\begin{proposition}
    \label{proposition:DefsOfHomoschismGroupsCoincide}%
    The definitions of the homoschism groups given in
    Definition~\ref{definition:ZerothAndFirstCohomoschismGroup} and
    Definition~\ref{definition:HomoschismGroups} coincide in the
    commutative, cocommutative case.
\end{proposition}
\begin{proof}
    The kernel of $\D^0$ is given by the coinvariants as
    $\D^0(f)(c) = (f \tensor \id)\coacts(c) = f(c) \tensor 1$ says
    $f(c^{(0)}) \tensor c^{(1)} = f(c) \tensor 1$ and hence
    $\coacts(c) = c \tensor 1$ follows, which is the definition of the
    coinvariants.

    For the first group, we use commutativity of $H$ and
    cocommutativity of $C$, to see that $\ker \D^1$ is given by
    elements of the form
    \begin{equation*}
        (\id \tensor \mu) \circ (f \tensor \id \tensor f) \circ
        (\coacts \tensor \id) \circ \Delta_C
        = \Delta_H \circ f,
    \end{equation*}
    where we also used the unitarity of $1$. So it coincides with the
    cocharacter group, as the comodule condition in this case is for
    free. The image of $\D^0$ are the inner cocharacters, i.e. those
    given by delta functionals.
\end{proof}

%
%

\section{Equivariant Morita-Takeuchi theory}
\label{sec:EquivariantMoritaTakeuchi}%

In noncommutative algebra it is a classical observation that not the
isomorphisms of algebras give rise to the most interesting category
but it is more convenient to extend the notion of morphism and
consider equivalence of module categories. Morita~\cite{morita:1958a}
showed that one can describe these equivalences via certain bimodules,
now termed Morita bimodules. For coalgebras attempts to consider
analogous notions where first done by Lin~\cite{Lin:1974} and more
fruitfully by Takeuchi~\cite{Takeuchi:1977}. The later approach, now
termed Morita-Takeuchi theory, turned out to be more convenient one
and was extended by Al-Takhman to coalgebras over rings
\cite{AlTakhman:2002a}.

Having a general notion of cosymmetry via the coaction of an Hopf
algebra it is natural to wonder how Morita-Takeuchi equivalence
changes in the presence of cosymmetries. Therefore we consider in this
section the notion of equivariant Morita-Takeuchi equivalence. This
section gives a justification for introducing the cocharacter groups
and the Schism complex, aswell, since these cohomology groups model
obstructions of lifting coactions on coalgebras to the corresponding
comodules. This shows that equivariant Morita-Takeuchi theory is
equivalently well-behaved as equivariant Morita theory, cf.\
\cite{jansen.waldmann:2006a}.

We recall the notion of an equivariant comodule.
\begin{definition}[$H$-equivariant comodule]
     \label{definition:EquivariantComodule}%
     Let $H$ be a Hopf algebra and $C$ be a coalgebra with coaction
     $\coacts$ of $H$. A (right) $C$-comodule $M$ is called
     $H$-equivariant comodule, if $\coacts$ lifts to $M$, i.e.\ there
     is a linear map $\coacts \colon M \to M \tensor H$, which turns
     $M$ into a $H$-comodule and additionally satisfies
     \begin{equation}
         \label{eq:EquivariantComodule}
         m_{(0)}^{(0)} \tensor m_{(1)}^{(0)} \tensor m^{(1)}
         = m_{(0)}^{(0)} \tensor m_{(1)}^{(0)} \tensor m_{(0)}^{(1)}
         m_{(1)}^{(1)}
     \end{equation}
     for all $m \in M$.
\end{definition}
The definition of a left equivariant comodule is analogously.

We need to check that equivariance is compatible with cotensorizing
comodules. For this recall that an $H$-module $M$ is $H$-pure in an
$H$-module $N$ if the natural map $M \tensor X \to N \tensor X$ is
injective for any $H$-module $X$. We remind the reader that the
cotensor product of two comodules is given by the equalizer of the
coactions, i.e. if $M,N$ are right and left $C$-comodules,
respectively, then the cotensor product $M \cotensor_C N$ is given by
\begin{equation}
    \label{eq:CotensorProduct}
        \begin{tikzpicture}[baseline = (current bounding box).center]
            \matrix(m)[matrix of math nodes,
            row sep = 3em,
            column sep = 5em]{
              M \cotensor[C] N & M \tensor[R] N & M \tensor_R C
              \tensor[R] N \\
            };
            \draw[dashed, ->] (m-1-1) to (m-1-2);
            \draw[->] ($(m-1-2.east) + (0,0.15)$) to
            node[above]{$\rho_M \tensor \id$} ($(m-1-3.west) +
            (0,0.15)$);
            \draw[->] ($(m-1-2.east) + (0,-0.15)$) to node[below]{$\id
              \tensor \rho_N$} ($(m-1-3.west) + (0,-0.15)$);
        \end{tikzpicture},
\end{equation}
where we denote by $\rho_M, \rho_N$ the corresponding coaction.
\begin{lemma}
    \label{item:EquivariantCotensorProduct}%
    Let $M,N$ be $H$-equivariant bicomodules over $C$, and let $M
    \cotensor N$ be $H$-pure in $M \tensor N$. Then the cotensor
    product $M \cotensor N$ is a $H$-equivariant comodule via
    \begin{equation}
        \label{eq:EquivariantCotensorProduct}
        (m \cotensor n)^{(0)} \tensor (m \cotensor n)^{(1)}
        = m^{(0)} \cotensor n^{(0)} \tensor m^{(1)} n^{(1)}
    \end{equation}
    for $m \in M, n \in N$.
\end{lemma}
\begin{proof}
    Note that $H$-purity is needed to make the definition well-defined
    as this gives the associativity up to isomorphism of the
    combination of the cotensor with the tensor product, see
    \cite[Lemma~2.3]{AlTakhman:2002a}. Let $m \in M, n \in N$, then
    \begin{align*}
        (m \cotensor n)_{(0)}^{(0)} \tensor (m \cotensor
        n)_{(1)}^{(0)} \tensor (m \cotensor n)^{(1)}
        &= (m \cotensor n)_{(0)}^{(0)} \tensor n_{(1)}^{(0)} \tensor
           m^{(1)} n^{(1)} \\
        &= m^{(0)} \tensor n_{(0)}^{(0)} \tensor n_{(1)}^{(0)} \tensor
           m^{(1)} n^{(1)} \\
        &= (\id^{(3)} \tensor \mu) \circ \tau^{(2 3)(4 5)} (m^{(0)} \tensor
           n_{(0)}^{(0)} \tensor n_{(1)}^{(0)} \tensor  n^{(1)}
           \tensor m^{(1)}) \\
        &= (\id^{(3)} \tensor \mu) \circ \tau^{(2 3)(4 5)} (m^{(0)} \tensor
           n_{(0)}^{(0)} \tensor n_{(1)}^{(0)} \tensor  n_{(0)}^{(1)}
           n_{(1)}^{(1)} \tensor m^{(1)}) \\
        &= m^{(0)} \tensor n_{(1)}^{(0)} \tensor n_{(0)}^{(0)}  \tensor
           m^{(1)} n_{(0)}^{(1)} n_{(1)}^{(1)} \\
        &= m_{(0)}^{(0)} \tensor m_{(1)}^{(0)} \tensor n^{(0)} \tensor
           m^{(1)} n_{(0)}^{(1)} n_{(1)}^{(1)} \\
        &= (\id^{(3)} \tensor \mu) \circ \tau^{(3 4)}( m_{(0)}^{(0)}
           \tensor m_{(1)}^{(0)} \tensor m^{(1)} \tensor n^{(0)}
           \tensor n_{(0)}^{(1)} n_{(1)}^{(1)}) \\
        &= (\id^{(3)} \tensor \mu) \circ \tau^{(3 4)}( m_{(0)}^{(0)}
           \tensor m_{(1)}^{(0)} \tensor m_{(0)}^{(1)} m_{(1)}^{(1)}
           \tensor n^{(0)} \tensor n_{(0)}^{(1)} n_{(1)}^{(1)}) \\
        &= (m \cotensor n)_{(0)}^{(0)} \tensor (m \cotensor
           n)_{(1)}^{(0)} \tensor m_{(0)}^{(1)} m_{(1)}^{(1)}
           n_{(0)}^{(1)} n_{(1)}^{(1)} \\
        &= (m \cotensor n)_{(0)}^{(0)} \tensor (m \cotensor
           n)_{(1)}^{(0)} \tensor (m \cotensor n)_{(0)}^{(1)} (m
           \cotensor n)_{(1)}^{(1)},
    \end{align*}
    and we have
    \begin{align*}
        (m \cotensor n)^{(0)} \tensor (m \cotensor n)^{(1)} \tensor (m
        \cotensor n)^{(2)}
        &= m^{(0)} \cotensor n^{(0)} \tensor m^{(1)} n^{(1)} \tensor
           m^{(2)} n^{(2)} \\
        &= m^{(0)} \cotensor n^{(0)} \tensor (m^{(1)} n^{(1)})^{(1)}
           \tensor (m^{(1)} n^{(1)})^{(2)} \\
        &= (m \cotensor n)^{(0)} \tensor ((m \cotensor n)^{(1)})^{(1)}
           \tensor ((m \cotensor n)^{(1)})^{(2)}
    \end{align*}
    hence the lemma is proved since $M \cotensor N$ is a bicomodule as
    cotensor product of bicomodules.
\end{proof}

We also need a corresponding notion of morphisms and isomorphisms.
\begin{definition}[$H$-equivariant comodule morphism]
    \label{definition:CoactionEquivariantMorphism}%
    Let $C$ be a coalgebra with $H$-coaction and let $M,N$ be
    equivariant comodules over $C$. A comodule morphism $\phi \colon M
    \to N$ is called $H$-equivariant if
    \begin{equation}
        \label{eq:CoactionEquivariantMorphism}
        \phi(m)^{(0)} \tensor \phi(m)^{(1)}
        = \phi(m^{(0)}) \tensor m^{(1)}
    \end{equation}
    for all $m \in M$.
    Two $H$-equivariant comodules $(M, \coacts), (N, \tilde{\coacts})$
    are isomorphic, if there exists a $\phi \in \Iso(M,N)$ such that
    $\tilde{\coacts} = \coacts_\phi$ with
    \begin{equation}
        \label{eq:IsomorphicEquivariantComodules}
        \coacts_\phi(n) = \phi(\phi^{-1}(n)^{(0)}) \tensor
        \phi^{-1}(n)^{(1)}
    \end{equation}
    for $n \in N$.
\end{definition}
It is obvious that the defined coaction is actually a coaction.

As we are interested in an equivariant Morita-Takeuchi theory, we
define now equivariant Morita-Takeuchi bicomodules. Recall from
\cite{AlTakhman:2002a} that for coalgebras over rings the equivalence
of their comodule categories is given by Morita-Takeuchi bicomodules.
For the definition of a Morita-Takeuchi bicomodule recall that a
$(C,D)$-bicomodule $M$ is called quasi-finite, if the functor
$\cargument \cotensor M \colon \Rcomod[C] \to \Rcomod[D]$ between the
comodule categories of $C$ and $D$ has a left adjoint, which will be
denoted by $\Cohom(M, \cargument)$, see \cite[Theorem
3.6]{AlTakhman:2002b} for the justification of this unfamiliar
definition of quasi-finiteness. If the cohom functor
$\Cohom(M, \cargument)$ is exact, one calls $M$ an injector. The
coendomorphism coalgebra is $\Coend(M) = \Cohom(M, M)$. A right
$C$-comodule is called faithfully coflat if the functor
$M \cotensor \cargument \colon \Rcomod[C] \to \Lmod[R]$ is exact and
faithful.

For two coalgebras $C$ and $D$ a Morita-Takeuchi bicomodule is a
$(C,D)$-bicomodules $M$, which is quasi-finite, faithfully coflat, an
injector in $\Rcomod[D]$ and one has $\Coend(M) \cong C$ naturally as
coalgebras. In the sequel we denote by the subscript $\argument_H$
the corresponding notion in the category of $H$-equivariant
bicomodules.
\begin{definition}[$H$-equivariant Morita-Takeuchi bicomodule]
    \label{definition:EquivariantMoritaTakeuchiBicomodule}%
    Let $C$ and $D$ be coalgebras with $H$-coaction. An equivariant
    $(C,D)$-bicomodule is called equivariant Morita-Takeuchi bicomdule
    if it is quasi-finite, faithfully coflat and an injector in
    $\Rcomod[D]_H$ and $\Coend_H(M) \cong C$ naturally as coalgebras
    via the left comodule structure.
\end{definition}
Next we investigate the notion of a twirled coaction.
\begin{theorem}
    \label{theorem:TwirledComodule}%
    Let $(M, \coacts)$ be a $H$-equivariant comodule and let $\phi \in
    \Hom(C,H)$. The twirled coaction $\coacts_\phi$ defined by
    \begin{equation}
        \label{eq:TwirledCoactionMorphism}
        \coacts_\phi
        = (\id \tensor \mu) \circ (\coacts \tensor \phi) \circ \rho
    \end{equation}
    i.e.
    \begin{equation}
        \label{eq:TwirledCoaction}
        \coacts_\phi(m)
        = m^{\phi(0)} \tensor m^{\phi(1)}
        = m_{(0)}^{(0)} \tensor m_{(0)}^{(1)} \phi(m_{(1)})
    \end{equation}
    for $m \in M$ defines another $H$-equivariant comodule $(M,
    \coacts_\phi)$ if $\phi \in \Gamma(C \Coacts H)$. We get
    ``iff'' if $M$ is an equivariant Morita-Takeuchi bicomodule.
\end{theorem}
\begin{proof}
    If $\phi \in \Gamma(C \Coacts H)$, we have to check that this
    really defines a coaction, that is we have to show it is a
    $H$-comodule, i.e. it fulfills
    \begin{equation*}
        m^{\phi(0)} \tensor m^{\phi(1)} \tensor m^{\phi(2)}
        = m^{\phi(0)} \tensor (m^{\phi(1)})^{(1)} \tensor
          (m^{\phi(1)})^{(2)}
    \end{equation*}
    and
    \begin{equation*}
        m \tensor 1 = m^{\phi(0)} \tensor \counit(m^{\phi(1)}),
    \end{equation*}
    and is compatible with the coaction on the coalgebra
    \begin{equation*}
        m_{(0)}^{\phi(0)} \tensor m_{(1)}^{\phi(0)} \tensor m^{\phi(1)}
        = m_{(0)}^{\phi(0)} \tensor m_{(1)}^{(0)} \tensor
          m_{(0)}^{\phi(1)} m_{(1)}^{(1)}.
    \end{equation*}

    For the comodule condition we calculate
    \begin{align*}
        m_{(0)}^{\phi(0)} \tensor m_{(1)}^{(0)} \tensor
        m_{(0)}^{\phi(1)} m_{(1)}^{(1)}
        &= m_{(0)}^{(0)} \tensor m_{(2)}^{(0)} \tensor m_{(0)}^{(1)}
           \phi(m_{(1)}) m_{(2)}^{(1)} \\
        &\stackrel{(a)}{=} m_{(0)}^{(0)} \tensor m_{(1)}^{(0)} \tensor
           m_{(0)}^{(1)} m_{(1)}^{(1)} \phi(m_{(2)}) \\
        &\stackrel{(b)}{=} m_{(0)}^{(0)} \tensor m_{(1)}^{(0)} \tensor
           m_{(0)}^{(1)} \phi(m_{(2)}) \\
        &= m_{(0)}^{\phi(0)} \tensor m_{(1)}^{\phi(0)} \tensor m^{\phi(1)},
    \end{align*}
    where we used in ($a$) the comodule condition of $\phi$, in ($b$)
    we used that $\coacts$ is compatible coaction on the comodule $M$,
    where we have to keep in mind that we used the $C$-comodule
    coaction two times which we do not see in the Sweedler notation.

    The counitarity of $\coacts_\phi$ is clear as everything involved
    is counitary. For the $H$-coaction condition we calculate
    \begin{align*}
        m^{\phi(0)} \tensor (m^{\phi(1)})^{(1)} \tensor
        (m^{\phi(1)})^{(2)}
        &= m_{(0)}^{(0)} \tensor m_{(0)}^{(1)} \phi(m_{(1)})^{(1)}
           \tensor m_{(0)}^{(2)} \phi(m_{(1)})^{(2)} \\
        &\stackrel{(a)}{=} m_{(0)}^{(0)} \tensor m_{(0)}^{(1)}
           \phi(m_{(1)}^{(0)}) \tensor m_{(0)}^{(2)} m_{(1)}^{(1)}
           \phi(m_{(2)}) \\
        &= m_{(0)}^{(0)} \tensor m_{(0)}^{(1)} \phi(m_{(1)}^{(0)})
           \tensor m_{(0)}^{(2)} m_{(0)}^{(1)} \counit(m_{(1)}^{(1)})
           \phi(m_{(2)}) \\
        &\stackrel{(b)}{=} m_{(0)}^{(0)} \tensor m_{(0)}^{(1)}
           \phi(m_{(1)}^{(0)}) \tensor m_{(0)}^{(1)}
           \counit(m_{(1)}^{(1)}) \phi(m_{(2)}) \\
        &= m_{(0)}^{(0)} \tensor m_{(0)}^{(1)} \phi(m_{(1)}^{(0)})
           \tensor \counit(m_{(0)}^{(1)}) m_{(0)}^{(2)}
           \counit(m_{(1)}^{(1)}) \phi(m_{(2)}) \\
        &\stackrel{(c)}{=} m_{(0)}^{(0)} \tensor m_{(0)}^{(1)}
           \phi(m_{(1)}^{(0)}) \tensor m_{(0)}^{(2)}
           \phi(\counit(m_{(1)}) m_{(2)}) \\
        &= m_{(0)}^{(0)} \tensor m_{(0)}^{(1)} \phi(m_{(1)}^{(0)})
           \tensor m_{(0)}^{(2)} \phi(m_{(1)}) \\
        &= m^{\phi(0)} \tensor m^{\phi(1)} \tensor m^{\phi(2)},
    \end{align*}
    where in ($a$) we used the coaction condition for $\phi$, in ($b$)
    we used that $\coacts$ is a compatible comodule coaction, in ($c$)
    the counitarity of $\phi$, and of course many times the properties
    of the counit.

    Now let $\coacts_\phi$ be a compatible coaction on $M$.
    Counitarity is again clear. We calculate
    \begin{align*}
        m_{(0)}^{(0)} \tensor m_{(0)}^{(1)} \phi(m_{(1)}^{(0)})
        \tensor m_{(0)}^{(2)} m_{(1)}^{(1)} \phi(m_{(2)})
        &= m_{(0)}^{(0)} \tensor m_{(0)}^{(1)} \phi(m_{(1)}^{(0)})
           \tensor m_{(0)}^{(1)} \phi(m_{(1)}) \\
        &= m^{\phi(0)} \tensor m^{\phi(1)} \tensor m^{\phi(2)} \\
        &= m^{\phi(0)} \tensor (m^{\phi(1)})^{(1)} \tensor
           (m^{\phi(1)})^{(2)} \\
        &= m_{(0)}^{(0)} \tensor m_{(0)}^{(1)} \phi(m_{(1)})^{(1)}
           \tensor m_{(0)}^{(2)} \phi(m_{(1)})^{(2)}
    \end{align*}
    and by $M$ being a Morita-Takeuchi bicomodule the coaction
    condition for $\phi$ follows.

    For the comodule condition we have
    \begin{align*}
        m_{(0)}^{(0)} \tensor m_{(1)}^{(0)} \tensor m_{(0)}^{(1)}
        m_{(1)}^{(1)} \phi(m_{(2)})
        &= m_{(0)}^{(0)} \tensor m_{(1)}^{(0)} \tensor m_{(0)}^{(1)}
           \phi(m_{(2)}) \\
        &= m_{(0)}^{\phi(0)} \tensor m_{(1)}^{\phi(0)} \tensor
           m^{\phi(1)} \\
        &= m_{(0)}^{\phi(0)} \tensor m_{(1)}^{(0)} \tensor
           m_{(0)}^{\phi(1)} m_{(1)}^{(1)} \\
        &= m_{(0)}^{(0)} \tensor m_{(2)}^{(0)} \tensor m_{(0)}^{(1)}
           \phi(m_{(1)}) m_{(2)}^{(1)}
    \end{align*}
    were the comodule condition for $\phi$ again follows by $M$ being a
    Morita-Takeuchi bicomodule.
\end{proof}
We are interested in those cocharacters which give isomorphic
equivariant Morita-Takeuchi bicomodules.
\begin{lemma}
    \label{lemma:CoinvariantsYieldIsomorphicEquivComodules}%
    Let $\phi \in \Gamma(C \Coacts H)$. Then $(M, \coacts)$ and $(M,
    \coacts_\phi)$ are isomorphic iff $\phi = \injco{c}$ for $c \in
    \Zentrum(C)$, where $\injco{\argument} \colon \Zentrum(C)
    \to \Gamma(C \Coacts H)$ is again given by
    \begin{equation}
        \label{eq:CoinvariantsYieldIsomorphicEquivComodules}
        \injco{c} = ((\delta_c \tensor \id) \circ \coacts) \conv
        (\delta_c \tensor 1)^{\conv -1}.
    \end{equation}
    Also $\coacts_{\injco{c}} = \coacts$ iff $c \in \Zentrum(C)^{\co
      H}$.
\end{lemma}
\begin{proof}
    In this situation we have for $\coacts_{\injco{c}}$ the following
    \begin{align*}
        m^{\injco{c}(0)} \tensor m^{\injco{c}(1)}
        &= m_{(0)}^{(0)} \tensor m_{(0)}^{(1)} \delta_c(m_{(1)}^{(0)})
           \delta_c^{\conv -1}(m_{(2)}) m_{(1)}^{(1)} \\
        &= m_{(0)}^{(0)} \tensor m^{(1)} \psi_c(m_{(1)}^{(0)}) \\
        &= \psi_c(m_{(1)}^{(0)}) m_{(0)}^{(0)} \tensor m^{(1)}
    \end{align*}
    with $\psi_c(m_{(1)}^{(0)}) = \delta_c(m_{(1)}^{(0)})
    \delta_c^{\conv -1}(m_{(2)}^{(0)}) \in R$ using counitarity. Now
    this is of the desired form, as every comodule morphism is by
    linearity of such a form for $c \in \Zentrum(C)$. Also one has
    $\psi(m_{(1)}^{(0)}) = 1$ iff $c^{(0)} \tensor c^{(1)} = c \tensor
    1$, which is the definition of the coinvariants.
\end{proof}
Hence we have the justification of the term homoschism, as every
element of the first homoschism group yields a new equivariant
comodule. Furthermore all of them arise by such elements.
\begin{proposition}
    \label{proposition:EveryCoactionOnComoduleArisesViaCohomoschism}%
    Let $(M, \coacts)$ be a $H$-equivariant Morita-Takeuchi
    bicomodule. Then the group $S^1(H,C)$ acts free and transitive on
    the set of all $H$-coactions, which split $M$ into a
    $H$-equivariant comodule. The group action is given by
    \begin{equation}
        \label{eq:EveryCoactionOnComoduleArisesViaCohomoschism}
        (\phi, \coacts) \mapsto \coacts_\phi
    \end{equation}
    for $\phi \in S^1(H,C)$.
\end{proposition}
\begin{proof}
    We only have to show, that this is really a group action,
    i.e. $\coacts_{\phi \conv \psi} = (\coacts_\phi)_\psi$. This is
    seen by the following calculation
    \begin{align*}
        m^{\phi \conv \psi(0)} \tensor m^{\phi \conv \psi(1)}
        &= m_{(0)}^{(0)} \tensor m_{(0)}^{(1)} (\phi \conv
           \psi)(m_{(1)}) \\
        &= m_{(0)}^{(0)} \tensor m_{(0)}^{(1)} \phi(m_{(1)})
           \psi(m_{(2)}) \\
        &= m_{(0)}^{\phi(0)} \tensor m_{(0)}^{\phi(1)} \psi(m_{(1)})
           \\
        &= m^{\psi(\phi(0))} \tensor m^{\psi(\phi(1))}.
    \end{align*}
    Hence the proposition follows.
\end{proof}
\begin{remark}
    \label{remark:ArgumentsWorkForLeftTwirledCoaction}%
    Note that all arguments work equivalently if one considers the
    left twirled coaction $\coacts^\phi$ given by
    \begin{equation}
        \label{eq:ArgumentsWorkForLeftTwirledCoaction}
        \coacts^\phi(m)
        = m_{(0)} \tensor \phi(m_{(1)}^{(0)}) m_{(1)}^{(1)}
    \end{equation}
    Hence the classification for left twirled coactions is the same as
    for right twirled coactions.
\end{remark}

We define the equivariant Picard groupoid as equivalence classes
of equivariant Morita bimodules and denote it by
\begin{equation}
    \label{eq:EquivariantCoPicardGroupoid}
    \Pic_H = \{ [M] \; | \; M \text{ is } H \text{-equivariant
      Morita-Takeuchi bicomodule} \}.
\end{equation}
We get an exact sequence connecting $\Pic_H$ and $\Aut_H$,
cf. Definition~\ref{definition:CoactionEquivariantMorphisms}:
\begin{equation}
    \label{eq:EquivariantCoPicardExactSequence}
    \begin{tikzpicture}[baseline = (current bounding box).center]
        \matrix(m)[matrix of math nodes,
        row sep = 3em,
        column sep = 5em]{
          1 & \InnAut_H(C) & \Aut_H(C) & \Pic_H(C) \\
        };
        \draw[->] (m-1-1) to (m-1-2);
        \draw[->] (m-1-2) to (m-1-3);
        \draw[->] (m-1-3) to node[above]{$\omega_H$} (m-1-4);
    \end{tikzpicture}
\end{equation}

We can describe the connection between $\Pic_H \to \Pic$ with the
group $S^1(H,C)$:
\begin{proposition}
    \label{proposition:MoritaTakeuchiInvarianceHomoschism}%
    Let $C,D$ be two Morita-Takeuchi equivalent coalgebras. Then
    either
    \begin{equation}
        \label{eq:MoritaTakeuchiInvarianceHomoschism}
        S^1(H,C) \cong S^1(H,D)
    \end{equation}
    or
    \begin{equation}
        \label{eq:EquivPicardGroupEmpty}
        \Pic_H(C,D) = \emptyset.
    \end{equation}
\end{proposition}
\begin{proof}
    If $\Pic_H(C,D) \not= \emptyset$ then $S^1(H,C)$ and $S^1(H,D)$
    are both acting free and transitive on $\Pic_H(C,D)$ and are
    commuting and hence must coincide.
\end{proof}
\begin{example}
    \label{example:EquivariantMoritaTakeuchiEquivComatrixCoalgebra}%
    Consider a coalgebra $C$ with coaction of a Hopf algebra $H$. Then
    we can define the coactions on $\Mat_n^*(C)$, the matrix
    coalgebra, and $C^n$ pointwise and get an equivariant
    Morita-Takeuchi equivalence between $\Mat_n^*(C)$ and $C$.
\end{example}
\begin{remark}
    \label{remark:DualityBetweenEquivCoalgebraAndAlgebras}%
    Having in mind that every coalgebra dualizes to an algebra, one
    can study the morphism $\Pic(C) \to \Pic(C^*)$ and try
    to describe, which parts of the $\Pic(C^*)$ are missing. Some results
    in this direction can be found in \cite[Thm.
    2.9]{Cuadra.Rozas.Torrecillas:2000}. In the equivariant
    setting one would like to consider something like the
    morphism $\Pic_H(C) \to \Pic_{H^*}(C^*)$, where now one has to be
    aware that $H^*$ might not be a well-defined Hopf algebra. Some
    steps in this direction can be found in \cite[Chap.
    9]{Montgomery:1993}, where the restricted dual for Hopf algebras
    is investigated. Using these restricted duals one could also
    wonder, what one can see by the morphism
    $\Pic_{H^\circ}(\algebra{A}^\circ) \to \Pic_H(\algebra{A})$. We
    postpone these questions to later projects.
\end{remark}


\section{Acknowledgments}
\label{sec:Acks}%

The author is grateful to Stefan Waldmann for advice and kind
supervision during obtaining the results presented in this article for
the author's master thesis.

%
%

{
  \footnotesize
  \renewcommand{\arraystretch}{0.5}
  \bibliographystyle{chairx}
  \bibliography{hopf}
}

%
%

\ifdraft{\clearpage}
\ifdraft{\phantomsection}
\ifdraft{\addcontentsline{toc}{section}{List of Corrections}}
\ifdraft{\listoffixmes}

%
%

\end{document}
